\documentclass[11pt]{article}

\usepackage{geometry}
\usepackage{amsfonts}
\usepackage{graphicx}
\usepackage{epstopdf}
\usepackage{enumitem}
\usepackage{algorithm}
\usepackage{algorithmic}
\usepackage{amsmath,amsfonts,amssymb,mathrsfs}
\usepackage{amsthm}
\usepackage{subcaption}
\ifpdf
  \DeclareGraphicsExtensions{.eps,.pdf,.png,.jpg}
\else
  \DeclareGraphicsExtensions{.eps}
\fi

\usepackage{amsopn}
\usepackage{mathtools}
\usepackage{cleveref}
\Crefname{theorem}{Theorem}{Theorems}
\Crefname{proposition}{Proposition}{Propositions}
\Crefname{lemma}{Lemma}{Lemmas}
\Crefname{corollary}{Corollary}{Corollaries}
\Crefname{remark}{Remark}{Remarks}
\Crefname{example}{Example}{Examples}
\Crefname{algorithm}{Algorithm}{Algorithms}
\Crefname{figure}{Figure}{Figures}
\Crefname{equation}{}{}

\theoremstyle{definition}
\newtheorem{definition}{Definition}[section]
\newtheorem{theorem}[definition]{Theorem}
\newtheorem{proposition}[definition]{Proposition}
\newtheorem{lemma}[definition]{Lemma}
\newtheorem{corollary}[definition]{Corollary}
\newtheorem{remark}[definition]{Remark}
\newtheorem{example}[definition]{Example}

\DeclareMathOperator*{\pmin}{p-min}
\DeclareMathOperator{\tropconv}{tconv}
\DeclareMathOperator{\supp}{supp}

\begin{document}

\title{Combinatorial Algorithm for Tropical Linearly Factorized Programming}
\author{Yuki Nishida\thanks{Kyoto Prefectural University (E-mail: y-nishida@kpu.ac.jp)}    }

\date{}

\maketitle

\begin{abstract} 
The tropical semiring is an algebraic system with addition ``$\max$'' and multiplication ``$+$''. 
As well as in conventional algebra, linear programming in the tropical semiring has been developed.
In this study, we introduce a new type of tropical optimization problem, namely, tropical linearly factorized programming.
This problem involves minimizing the objective function given by a product of tropical linear forms divided by a tropical monomial, subject to tropical linear inequality constraints.
As the objective function is equivalent to the dual of the transportation problem, it is convex in the conventional sense but not in the tropical sense, while the feasible set is convex in the tropical sense but not in the conventional sense. 

Our algorithm for tropical linearly factorized programming is based on the descent method.
We first show that a feasible descent direction can be characterized in terms of a specific digraph, called a tangent digraph.
Especially in non-degenerate cases, we present a simplex-like algorithm that updates the tree structure of tangent digraphs iteratively.
Each iteration can be executed in $O(r_A+r_C)$ time, where $r_A$ and $r_C$ are the numbers of finite coefficients in the constraints and objective function, respectively.
For integer instances, our algorithm finds a local optimum in pseudo-polynomial time.
\end{abstract}

\section{Introduction}
The tropical semiring, or max-plus algebra, is a set of numbers $\mathbb{R}_{\max}:= \mathbb{R} \cup \{-\infty\}$ with addition $a \oplus b:= \max(a,b)$ and multiplication $a \otimes b:= a + b$.
In the tropical semiring, subtraction cannot be defined, whereas tropical division is defined by $a \oslash b:= a \otimes (-b)$ if $b \neq -\infty$.
As linear algebra over the tropical semiring is referred to as ``the linear algebra of combinatorics''~\cite {Butkovic2003}, it is closely related to combinatorial optimization problems.
For example, the maximum eigenvalue of a square matrix coincides with the maximum cycle mean in the associated weighted digraph~\cite{Green1979};
the power series of a square matrix is translated into the shortest path problem~\cite{Gondran1975};
the determinant of a square matrix is identical to the optimal assignment in a weighted bipartite graph~\cite{Bapat1995}.

A tropical linear programming (TLP) involves minimizing a tropical linear function of the form $c \otimes x \oplus d$ subject to a tropical linear constraint $A^+ \otimes x \oplus b^+ \geq A^- \otimes x \oplus b^-$ for decision variables $x \in \mathbb{R}_{\max}^n$, where $A^\pm \in \mathbb{R}_{\max}^{m \times n}$, $b^\pm \in \mathbb{R}_{\max}^m$, $c \in \mathbb{R}_{\max}^{1 \times n}$, and $d \in \mathbb{R}_{\max}$. 
This problem is found in early works of the ordered algebra~\cite{Zimmermann1981} or the extremal algebra~\cite{Zimmermann2003}, both of which include the tropical semiring, and in application areas such as scheduling with and/or constraints~\cite{Mohring2004}.
However, its algorithmic aspects have attracted particular attention following the work of Butkovi\v{c} and Aminu~\cite{Butkovic2009}.
We remark that finding even a feasible solution satisfying $A^+ \otimes x \oplus b^+ \geq A^- \otimes x \oplus b^-$ is not easy. 
The alternating method proposed by Cuninghame-Green and Butkovi\v{c}~\cite{Green2003} is known to solve this system of inequalities in pseudo-polynomial time complexity.
Akian et al.~\cite{Akian2010} demonstrated that finding a solution to a system of two-sided tropical linear inequalities is equivalent to solving the mean-payoff game, which is a well-known $\text{NP} \cap \text{co-NP}$ problem but has no known polynomial-time algorithm~\cite{Zwick1996}.
Hence, algorithms for the mean-payoff game, such as the GKK algorithm~\cite{Gurvich1988}, value iteration~\cite{Zwick1996}, policy iteration~\cite{Bjorklund2007,Dhingra2006}, and energy game approach~\cite{Benerecetti2024,Brim2011,Dorfman2019}, are also applicable to tropical linear systems.

In this study, we introduce a new type of objective function, namely, a tropical linearly factorized Laurent polynomial, of the form
\begin{align*}
	g(x) := \bigotimes_{k=1}^p (c_{k,1} \otimes x_1 \oplus \cdots \oplus  c_{k,n} \otimes x_n \oplus d_k)^{\otimes \mu^+_k} \oslash (x_1^{\otimes \mu^-_1} \otimes \cdots \otimes x_n^{\otimes \mu^-_n}).
\end{align*}
Here, the tropical power $a^{\otimes r}$ stands for the $r$-fold tropical product, that is, $a^{\otimes r} = ra$, and $x_j$ denotes the $j$th entry of vector $x$.
A tropical linear function is obtained by setting $p=1$, $\mu^+_1 = 1$, and $\mu^-_j = 0$ for $j=1,2,\dots,n$. 
For simplicity of notation, we study the homogenized version of tropical linearly factorized programming (TLFacP) defined as
\begin{alignat}{2}
&\text{minimize}& \qquad & f(x) := (C \otimes x)^{\otimes \mu^+} \oslash x^{\otimes \mu^-}, \label{eq:objective} \\
&\text{subject to}& \qquad & A^+ \otimes x \geq A^- \otimes x, 	\label{eq:const}
\end{alignat}
where $A^\pm = (a_{i,j}^{\pm}) \in \mathbb{R}_{\max}^{m \times n}$, $C = (c_{k,j}) \in \mathbb{R}_{\max}^{p \times n}$, $\mu^+ \in \mathbb{Z}_{+}^p$ and $\mu^- \in \mathbb{Z}_{+}^n$.
Note that inhomogeneous problems can be easily translated to homogeneous ones.
We adopt the abbreviation TLFacP to distinguish it from tropical linear fractional programming (TLFP)~\cite{Gaubert2012}.
In conventional expression, the TLFacP \eqref{eq:objective}--\eqref{eq:const} can be rewritten as
\begin{alignat*}{3}
&\text{minimize}& \qquad & \sum_{k=1}^p \mu^+_k \cdot \max_j (c_{k,j} + x_j) - \sum_{j=1}^n \mu^-_j x_j, &\qquad& \\
&\text{subject to}& \qquad & \max_j (a^+_{i,j} + x_j) \geq \max_j (a^-_{i,j} + x_j), &\qquad& i=1,2,\dots,m.
\end{alignat*}
When the constraint~\eqref{eq:const} is not imposed, the problem is equivalent to the dual of the Hitchcock transportation problem~\cite{Hitchcock1941} defined by the cost matrix $C$ and the supply and demand vectors $\mu^+$ and $\mu^-$, respectively.
Thus, TLFacP is a common generalization of TLP and the transportation problem.
In addition, the function~\eqref{eq:objective} is a polyhedral L-convex function~\cite{Murota2000} that possesses many desirable properties originating from discrete convexity~\cite{Murota1998}.
An L-convex polyhedron is known to be convex both conventionally and tropically~\cite{Johnson2015}.
These facts suggest that the TLFacP connects discrete and tropical convex analysis.

\begin{figure}
\centering
\includegraphics[width=0.5\textwidth]{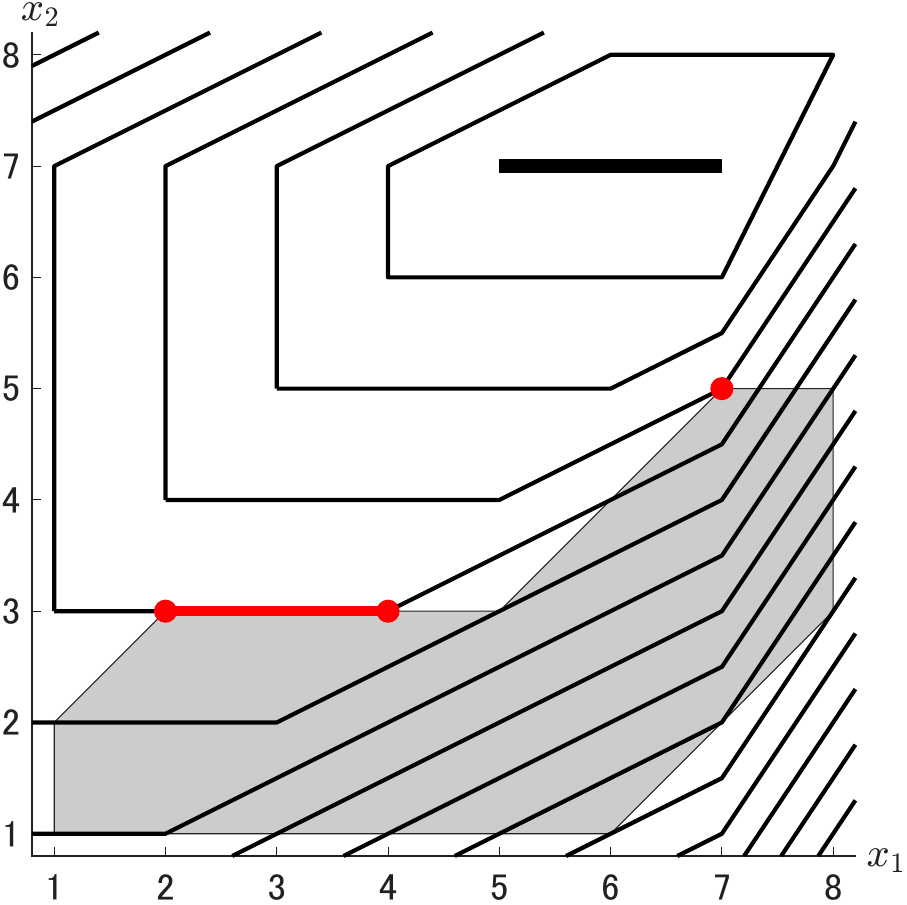}
\caption{Examples of contour lines of $f(x)$ and feasible set $S$ projected onto the plane $x_3 = 0$. The objective function $f(x)$ attains the minimum on the segment $\{(x_1,7,0) \mid 5 \leq x_1 \leq 7\}$, which does not intersect $S$. The optimal solution is $(7,5,0)$. All points on the segment $\{(x_1,3,0) \mid 2 \leq x_1 \leq 4\}$ are sub-optimal local optima.}		\label{fig:intro1}
\end{figure}

\subsection{Our Results}

The feature of TLFacP is summarized as follows:
(i) the objective function $f(x)$ is convex in the conventional sense, whereas the feasible set is not; 
(ii) the feasible set $S$ is convex in the tropical sense, whereas the objective function is not.
Hence, TLFacP exhibits characteristics of non-convex optimization in the sense that a local optimum may not be a global optimum, see \Cref{fig:intro1}.
In Section 3.4, we will see that TLFacP is NP-hard by constructing a reduction from the minimum set cover problem, which is known to be NP-hard~\cite{Karp1972}.
Finding an optimal solution is difficult; hence, we propose an algorithm to find a local optimum of a TLFacP.

Our algorithm is based on the descent method.
We first show that the descent direction is given by the indicator vector of some subset of $\{1,2,\dots,n\}$. This means that we can determine a descent direction in a combinatorial manner.
Thus, the core of our algorithm is to determine a subset corresponding to a feasible descent direction.
Next, we exploit the tangent digraph to analyze the change in the objective value when each entry of the decision variable is increased.
The concept of tangent digraphs was first introduced in~\cite{Allamigeon2013} to characterize the extreme rays of tropical convex cones and was later applied in the tropical simplex method for TLP~\cite{Allamigeon2015} to understand active constraints.
Tangent digraphs in these works only have edges corresponding to the constraints.
In this study, we extend the tangent digraph by appending edges corresponding to linear forms in the objective function so that we can deal with tropical linearly factorized polynomials rather than single tropical linear forms.
Then, we show that a feasible descent direction is determined by solving the minimum $s$-$t$ cut problem on some subgraph of the tangent digraph.

If the TLFacP \eqref{eq:objective}--\eqref{eq:const} is non-degenerate, the tangent digraph does not contain a circuit.
This fact enables us to efficiently determine a feasible descent direction as a tree obtained by removing an edge from the tangent digraph.
This leads to the polynomial time complexity $O(r_A+r_C)$ to find a feasible descent direction, where $r_A$ is the total number of finite entries in $A^+$ and $A^-$, and $r_C$ is the number of finite entries in $C$.
For integer instance, we show the total complexity to find a local optimum is at most $O((m+n)(r_A+r_C)MD)$, where $M$ is the maximum absolute value of the finite entries of $A^+$, $A^-$, and $C$, and $D = \sum_{k=1}^p \mu^+_k$.

We again note that no polynomial-time algorithm is known to find even a feasible solution of a TLFacP.
By adopting the mean-payoff game approach, a feasible solution is obtained by an $O((m+n)r_AM)$ algorithm in~\cite{Brim2011} that finds a winning strategy of the game.
Our algorithm to solve a non-degenerate TLFacP has the complexity obtained by almost multiplying the complexity of finding a feasible solution by the degree $D$ of the objective function.
In particular, the complexity of finding a local optimum is the same as that of finding a feasible solution when $D$ is a small fixed value.

\subsection{Related Works}

In the field of tropical optimization, TLP has attracted significant attention.
The first algorithm for TLP in Butkovi\v{c} and Aminu~\cite{Butkovic2009} was a bisectional search with $O(mn(m+n)M\log M)$ time complexity.
Some polynomial-time algorithms were developed in subsequent studies~\cite{Butkovic2014,Butkovic2015} for the very special case in which the fractional parts of the coefficients satisfy a property called OneFP. 
Allamigeon et al.~\cite{Allamigeon2015} proposed a tropical simplex method to solve a TLP by exploiting the correspondence between the tropical semiring and the field of the Hahn series~\cite{Develin2007}.
Subsequently, some variant algorithms have been developed, such as the shadow-vertex algorithm~\cite{Allamigeon2014a} and the extension to degenerate cases ~\cite{Allamigeon2014b}.
The simplex method was reformulated for the abstraction of TLP using a tropical analogue of oriented matroid~\cite{Loho2020}. 
Our proposed algorithm can be regarded as a variant of the tropical simplex method, but it applies to a broader class of objective functions than tropical linear ones.

Some studies on tropical optimization problems have addressed nonlinear objective functions.
For example, a linear fractional function $(c^+ \otimes x \oplus d^+) \oslash (c^- \otimes x \oplus d^-)$ was considered in~\cite{Gaubert2012}, where $c^+$ and $c^-$ are row vectors and $d^+$ and $d^-$ are scalars. Newton's method based on the parametric mean-payoff game was proposed to minimize the aforementioned function. In~\cite{Goncalves2014}, a tropical analogue of the Charnes--Cooper transformation was developed to transform a tropical linear fractional programming into a TLP.
A pseudo-quadratic objective function $q(x):= (-x^\top) \otimes C \otimes x \oplus (-x^\top) \otimes d^+ \oplus d^- \otimes x$ has been considered in a series of works by Krivulin, e.g.,~\cite{Krivulin2015,Krivulin2013,Krivulin2024}, and its relations to the tropical spectral theory and applications to the scheduling problem have been investigated. A bisection and Newton's methods for minimizing $q(x)$ above were developed in~\cite{Parsons2023}.
Compared to these problems, this study focuses on an objective function whose numerator and denominator have arbitrarily large degrees. 
Aminu and Butkovi\v{c}~\cite{Aminu2011} developed heuristic algorithms to minimize nonlinear isotone functions.
The main content of the aforementioned work is the local search changing only one or two entries of the current solution. 
In contrast, our algorithm does not restrict the size of neighborhoods; hence, it is less likely to fall into a local optimum.
Another optimization problem related to the tropical semiring is location problems based on tropical metric~\cite{Comaneci2024}.
Such problems include the minimization of our objective function~\eqref{eq:objective}, but no constraints are imposed.
Recently, the tropical gradient descent to solve the tropical location problem has been developed~\cite{Talbut2025}.

\subsection{Organization}

The rest of this paper is organized as follows.
In Section 2, we present definitions and notations on tropical linear algebra.
In Section 3, we describe some basic facts relating to TLFacP, such as the homogenization, the connection to the transportation problem, the L-convexity of the objective function, and NP-hardness.
In Section 4, we first characterize the local optimality of a TLFacP.
Subsequently, we describe the descent method to find a local optimum using the tangent digraph. 
In Section 5, we devise a pseudo-polynomial-time algorithm for the non-degenerate case and prove its complexity.
Finally, we summarize the results of the paper with some remarks in Section 6.

\section{Preliminaries on Tropical Linear Algebra}

First, the notations used in this study are described.
The set of all nonnegative integers is denoted by $\mathbb{Z}_{+}$.
The $j$th entry of vector $x$ is denoted by $(x)_j$, or simply by $x_j$.
For vectors $x$ and $y$, the inequality $x \geq y$ indicates that $x_j \geq y_j$ holds for all $j$.
Other types of inequalities can be understood similarly.
For positive integer $n$, the set of integers $\{1,2,\dots,n\}$ is denoted by $[n]$.
For $J \subset [n]$, the indicator vector $\chi_J$ is defined by $(\chi_J)_j = 1$ if $j \in J$ and $(\chi_J)_j = 0$ otherwise.

\subsection{Tropical Linear Algebra}

The tropical semiring is a set of numbers $\mathbb{R}_{\max} := \mathbb{R} \cup \{-\infty\}$ with tropical addition $a \oplus b := \max(a,b)$ and multiplication $a \otimes b := a+ b$.
Tropical division is defined as $a \oslash b:= a - b$.
Tropical power $a^{\otimes r}$ is the $r$-fold tropical product of $a$, which is identical to $ra$.
We state that $a \in \mathbb{R}_{\max}$ is finite when we emphasize that $a \neq -\infty$.
For details, please refer to~\cite{Baccelli1992,Butkovic2010,Heidergott2005,Joswig2021}.

Let $\mathbb{R}_{\max}^n$ and $\mathbb{R}_{\max}^{m \times n}$ be the sets of $n$-dimensional column vectors and $m \times n$ matrices whose entries are in $\mathbb{R}_{\max}$, respectively.
Arithmetic operations of vectors and matrices over $\mathbb{R}_{\max}$ are defined as in conventional linear algebra.
For vectors $x, y \in \mathbb{R}_{\max}^{n}$, a matrix $A = (a_{i,j}) \in \mathbb{R}_{\max}^{m \times n}$, and a scalar $\alpha \in \mathbb{R}_{\max}$, vector sum $x \oplus y \in \mathbb{R}_{\max}^{n}$, matrix-vector product $A \otimes x \in \mathbb{R}_{\max}^{m}$, and scalar multiplication $\alpha \otimes x \in \mathbb{R}_{\max}^{n}$ are defined as follows:
\begin{align*}
	&(x \oplus y)_{j} := x_j \oplus y_j = \max(x_j,y_j), \\
	&(A \otimes x)_i := \bigoplus_{j=1}^n a_{i,j} \otimes x_j = \max_{j \in [n]} (a_{i,j} + x_j), \\
	&(\alpha \otimes x)_j := \alpha \otimes x_j = \alpha + x_j.
\end{align*}
(Matrix sum and matrix-matrix product are also defined similarly, but we will not use them in this paper.)
For $x \in \mathbb{R}_{\max}^n$ and $\mu \in \mathbb{Z}_+^n$, multi-indexed power $x^{\otimes \mu}$ is defined by $x^{\otimes \mu} := \bigotimes_{j\in[n]} x_j^{\otimes \mu_j} = \sum_{j\in[n]} \mu_jx_j$.
The ``dual'' operator $a \oplus' b:= \min(a,b)$ is sometimes useful in tropical linear algebra.
For vectors $x, y \in \mathbb{R}_{\max}^{n}$, ``min-plus'' vector sum $x \oplus' y \in \mathbb{R}_{\max}^{n}$ is defined by
\begin{align*}
	(x \oplus' y)_{j} := x_j \oplus' y_j = \min(x_j,y_j).
\end{align*}

Let $\supp(x)$ denote the support of $x \in \mathbb{R}_{\max}^n$, that is, $\supp(x) = \{ j \in [n] : x_j \neq \infty\}$.
The tropical distance between two points $x,y \in \mathbb{R}_{\max}^n$ is defined by
\begin{align*}
	d(x,y) := \begin{cases} \max_{j \in \supp(x)} |x_j-y_j| & \text{if } \supp(x) = \supp(y), \\ \infty & \text{otherwise}. \end{cases}		
\end{align*}
This means that we cannot move from one vector to another with a different support by a continuous change.

A nonempty subset $X \subset \mathbb{R}_{\max}^n$ is called a subspace if $\alpha \otimes x \oplus \beta \otimes y \in X$ for all $x,y \in X$ and $\alpha,\beta \in \mathbb{R}_{\max}$.
A subspace of $\mathbb{R}_{\max}^n$ is also regarded as a tropical convex cone since all finite values are larger than $-\infty$, which is ``zero'' in the tropical semiring. 
Especially, $\tropconv(x,y) := \{ \alpha \otimes x \oplus \beta \otimes y : \alpha \oplus \beta = 0\}$ is called the tropical segment between $x$ and $y$.
Refer to \cite{Develin2004} for details of tropical convexity.

\subsection{Tropical linear systems}
For tropical matrices $A^+, A^- \in \mathbb{R}_{\max}^{m \times n}$, we consider a tropical linear system $A^+ \otimes x \geq A^- \otimes x$.
The set of all solutions is denoted by $S(A^+,A^-)$, that is, $S(A^+,A^-) := \{ x \in \mathbb{R}_{\max}^n : A^+ \otimes x \geq A^- \otimes x\}$.
The next proposition states that $S(A^+,A^-)$ is a subspace of $\mathbb{R}_{\max}^n$.

\begin{proposition}[e.g., \cite{Green1979}]		\label{prop:linsys}
The set of all finite solutions $S(A^+,A^-)$ satisfies the following properties:
\begin{enumerate}
\item $\alpha \otimes x \in S(A^+,A^-)$ for any $x \in S(A^+,A^-)$ and $\alpha \in \mathbb{R}_{\max}$.
\item $x \oplus y \in S(A^+,A^-)$ for any $x,y \in S(A^+,A^-)$.
\end{enumerate}
\end{proposition}
If all entries of $A^+$ and $A^-$ are finite integers, then an integer solution of $A^+ \otimes x \geq A^- \otimes x$ can be found in at most $O(mn(m+n)M)$ time using the alternating method~\cite{Green2003},
where $M$ is the maximum absolute value of the integer entries of $A^+$ and $A^-$.
Even if $A^+$ and/or $A^-$ contain (non-integer) real numbers, the alternating method terminates within a finite number of iterations~\cite{Sergeev2009}.
Moreover, if we use the result in the mean-payoff game, a feasible solution can be obtained in $O((m+n)r_AM)$ time, where $r_A$ is the total number of finite entries of $A^+$ and $A^-$~\cite{Brim2011}.

\section{Basic Results on Tropical Linearly Factorized Programming}

In this section, we present some results on TLFacP to situate it in related problems.

\subsection{Inhomogeneous Problem}

Let us consider an inhomogeneous TLFacP of the form
\begin{alignat}{2} 
&\text{minimize}& \qquad & (C \otimes x \oplus d)^{\otimes \mu^+} \oslash x^{\otimes \mu^-}, 	\label{eq:inhom1}\\
&\text{subject to}& \qquad & A^+ \otimes x \oplus b^+ \geq A^- \otimes x \oplus b^-, 	\label{eq:inhom2}
\end{alignat}
where $A^\pm \in \mathbb{R}_{\max}^{m \times n}$, $b^\pm \in \mathbb{R}_{\max}^{m}$, $C \in \mathbb{R}_{\max}^{p \times n}$, $d \in \mathbb{R}_{\max}^{p}$, $\mu^+ \in \mathbb{Z}_{+}^p$ and $\mu^- \in \mathbb{Z}_{+}^n$. 
Introducing a new variable $x_{n+1}$ and defining $\nu := \sum_{j \in [n]} \mu^-_j - \sum_{k \in [p]} \mu^+_k$, we homogenize this problem as
\begin{alignat}{2}
&\text{minimize}& \qquad & (\tilde{C} \otimes \tilde{x})^{\otimes \tilde{\mu}^+} \oslash \tilde{x}^{\otimes \tilde{\mu}^-}, 	\label{eq:hom1}\\
&\text{subject to}& \qquad & \tilde{A}^+ \otimes \tilde{x} \geq \tilde{A}^- \otimes \tilde{x},	\label{eq:hom2}
\end{alignat}
where $\tilde{x} = \begin{psmallmatrix} x \\ x_{n+1} \end{psmallmatrix}$, $\tilde{A}^{\pm} = \begin{pmatrix} A^{\pm} & b^{\pm} \end{pmatrix}$, $\tilde{C} = \begin{psmallmatrix} &C& & d \\ -\infty &\cdots&-\infty & 0 \end{psmallmatrix}$, $\tilde{\mu}^+ = \begin{psmallmatrix} \mu^+ \\ \max(\nu,0) \end{psmallmatrix}$, and $\tilde{\mu}^- = \begin{psmallmatrix} \mu^- \\ \max(-\nu,0) \end{psmallmatrix}$.
We remark that $\sum_{k \in [p+1]} \tilde{\mu}^+_k = \sum_{j \in [n+1]} \tilde{\mu}^-_j$.
Then, $x \in \mathbb{R}_{\max}^n$ satisfies~\eqref{eq:inhom2} if and only if $\begin{psmallmatrix} x \\ 0 \end{psmallmatrix} \in S(\tilde{A}^+, \tilde{A}^-)$. Moreover, by \Cref{prop:linsys}, $(-x_{n+1}) \otimes x$ satisfies~\eqref{eq:inhom2} for any $\tilde{x} \in S(\tilde{A}^+, \tilde{A}^-)$ with $x_{n+1} \neq -\infty$.
We observe that
\begin{align*}
	(\tilde{C} \otimes \tilde{x})^{\otimes \tilde{\mu}^+} \oslash \tilde{x}^{\otimes \tilde{\mu}^-} 
	& = (\tilde{C} \otimes ((-x_{n+1}) \otimes \tilde{x}))^{\otimes \tilde{\mu}^+} \oslash ((-x_{n+1}) \otimes \tilde{x})^{\otimes \tilde{\mu}^-} \\
	& = (C \otimes x \oplus d)^{\otimes \mu^+} \oslash x^{\otimes \mu^-}.
\end{align*}
Thus, the inhomogeneous problem~\eqref{eq:inhom1}--\eqref{eq:inhom2} is equivalent to the homogenized problem~\eqref{eq:hom1}--\eqref{eq:hom2}.

\subsection{Connection to Transportation Problem}

Here, we observe that the minimization of $f(x)$ defined in~\eqref{eq:objective} is equivalent to the dual problem of the Hitchcock transportation problem~\cite{Hitchcock1941}.
Let us consider a directed bipartite graph with vertex sets $U = \{u_1,u_2,\dots,u_p\}$ and $V = [n]$, and edge set $E = \{(u_k, j) \in U \times V : c_{k,j} \neq -\infty\}$. The cost of an edge $(u_k,j) \in E$ is given by $c_{k,j}$. For every edge, the lower and upper bounds of the capacity are $0$ and $\infty$, respectively. The supply of $u_k\in U$ and demand of $j \in V$ are $\mu^+_k$ and $\mu^-_j$, respectively. Then, the maximum-cost flow problem on this bipartite graph can be represented by the following linear programming problem:
\begin{alignat*}{3}
&\text{maximize}& \qquad & \sum_{(u_k,j) \in E} c_{k,j} z_{k,j}, & \quad &\\
&\text{subject to}& \qquad & \sum_{j : (u_k,j) \in E} z_{k,j} = \mu^+_k, &\quad& k \in [p], \\
& & \qquad & \sum_{k: (u_k,j) \in E} z_{k,j} = \mu^-_i, &\quad& j \in [n], \\
& & \qquad & z_{k,j} \geq 0, &\quad& (u_k,j) \in E. 
\end{alignat*}
The dual of this problem is derived as follows:
\begin{alignat*}{3}
&\text{minimize}& \qquad & \sum_{k=1}^p \mu^+_k y_k - \sum_{j=1}^n \mu^-_j x_j, & \quad &\\
&\text{subject to}& \qquad & y_k - x_j \geq c_{k,j}, &\quad& (u_k,j) \in E. 
\end{alignat*}
By recalling that $(u_k,j) \not\in E$ implies $c_{k,j} = -\infty$, the constraints of the dual problem are written as
\begin{align*}
	y_k \geq \max_{j \in [n]} (c_{k,j} + x_j) = (C \otimes x)_k, \qquad k \in [p].
\end{align*}
As long as we are to minimize the objective function, we may set $y_k = (C\otimes x)_k$ for all $k \in [p]$.
Then, we obtain the following unconstrained minimization problem:
\begin{alignat*}{2}
&\text{minimize}& \qquad & f(x) = (C \otimes x)^{\otimes \mu^+} \oslash x^{\otimes \mu^-}.
\end{alignat*}
If $\sum_{k \in [p]} \mu^+_k \neq \sum_{j \in [n]} \mu^-_j$, then the objective function $f(x)$ is not bounded below because the primal problem has no feasible solution.

\subsection{L-convexity of Objective Function}

We next focus on the convexity of $f(x)$ given by~\eqref{eq:objective}.
We observe that $f(x)$ is a polyhedral convex function because it is expressed as the maximum of finitely many affine functions.
\Cref{prop:lconv0} below implies that $f(x)$ is a polyhedral L-convex function~\cite{Murota2000}.
Although this is a consequence of Section 3.2 together with general results in~\cite{Murota2000}, we present a self-contained proof.

\begin{proposition}		\label{prop:lconv0}
The function $f(x)$ defined by~\eqref{eq:objective} satisfies the following properties.
\begin{enumerate}
\item $f(\alpha \otimes x) = f(x) + \alpha\,(\sum_{k \in [p]} \mu^+_k - \sum_{j \in [n]} \mu^-_j)$ for any $x \in \mathbb{R}^n_{\max}$ and $\alpha \in \mathbb{R}$.
\item $f(x) + f(y) \geq f(x \oplus y) + f(x \oplus' y)$ for any $x,y \in \mathbb{R}_{\max}^n$.
\end{enumerate}
\end{proposition}

\begin{proof}For $x \in \mathbb{R}_{\max}^n$ and $\alpha \in \mathbb{R}$, we have
\begin{align*}
	f(\alpha \otimes x) &= \sum_{k \in [p]} \mu^+_k \alpha + \sum_{k \in [p]} \mu^+_k \max_{j \in [n]} (c_{k,j} + x_j) - \sum_{j \in [n]} \mu^-_j \alpha -\sum_{j \in [n]} \mu^-_j  x_j \\
	&= f(x) + \alpha \left( \sum_{k \in [p]} \mu^+_k - \sum_{j \in [n]} \mu^-_j \right),
\end{align*}
obtaining Property 1.

Since $x_j + y_j = \max(x_j,y_j) + \min(x_j,y_j)$ for all $j \in [n]$, we can verify that 
\begin{align}
	\sum_{j \in [n]} \mu^-_j x_j + \sum_{j \in [n]} \mu^-_j y_j = \sum_{j \in [n]} \mu^-_j (x \oplus y)_j + \sum_{j \in [n]} \mu^-_j (x \oplus' y)_j.	\label{eq:lconv1}
\end{align}
For each $k \in [p]$, we may assume without loss of generality that $(C \otimes (x \oplus y))_k = (C \otimes x)_k$ because we have $(C \otimes (x \oplus y))_k = (C \otimes x)_k \oplus (C \otimes y)_k$.
Noting that $C \otimes y \geq C \otimes (x \oplus' y)$, we have
\begin{align*}
	(C \otimes x)_k + (C \otimes y)_k \geq (C \otimes (x \oplus y))_k + (C \otimes (x \oplus' y))_k.
\end{align*}
As this holds for all $k \in [p]$, we have
\begin{align*}
	&\sum_{k \in [p]} \mu^+_k (C \otimes x)_k + \sum_{k \in [p]} \mu^+_k (C \otimes y)_k  \\
	&\qquad \geq \sum_{k \in [p]} \mu^+_k (C \otimes (x\oplus y))_k + \sum_{k \in [p]} \mu^+_k (C \otimes (x \oplus' y))_k.	
\end{align*}
By combining this inequality with \eqref{eq:lconv1}, we obtain Property 2.
\end{proof}

If $\sum_{k \in [p]} \mu^+_k \neq \sum_{j \in [n]} \mu^-_j$, TLFacP \eqref{eq:objective}--\eqref{eq:const} has no optimal solution because of \Cref{prop:linsys,prop:lconv0}.
Hence, we will consider the case where $\sum_{k \in [p]} \mu^+_k = \sum_{j \in [n]} \mu^-_j$. This condition is satisfied for the problem obtained by homogenization of an inhomogeneous problem.
\Cref{prop:lconv0} will be referred to in the following form.

\begin{corollary}		\label{prop:lconv}
If $\sum_{k \in [p]} \mu^+_k = \sum_{j \in [n]} \mu^-_j$, the function $f(x)$ defined by~\eqref{eq:objective} satisfies the following properties.
\begin{enumerate}
\item $f(\alpha \otimes x) = f(x)$ for any $x \in \mathbb{R}_{\max}^n$ and $\alpha \in \mathbb{R}$.
\item $f(x) + f(y) \geq f(x \oplus y) + f(x \oplus' y)$ for any $x,y \in \mathbb{R}_{\max}^n$.
\end{enumerate}
\end{corollary}

\subsection{NP-hardness}

Although the objective function is a polyhedral L-convex function, finding a global optimum of a TLFacP is still difficult because the feasible set is not convex in the sense of conventional algebra. Here, we show that TLFacP is NP-hard by constructing a polynomial reduction from the minimum set cover problem. A similar technique can be found in~\cite{Butkovic2009}.
Given a ground set $W$ and a family of subsets $\mathscr{I} = \{ I_1, I_2, \dots, I_n \} \subset 2^W$, a subfamily $\{I_{j_1}, I_{j_2}, \dots, I_{j_q}\} \subset \mathscr{I}$ is called a set cover if $I_{j_1} \cup I_{j_2} \cup \cdots \cup I_{j_q} = W$.
The minimum set cover is the one that has the smallest number of subsets.
As the set cover problem is NP-complete~\cite{Karp1972}, the minimum set cover problem is NP-hard.

\begin{theorem}
TLFacP is NP-hard.
\end{theorem}

\begin{proof}
Consider any instance $\mathscr{I} = \{ I_1, I_2, \dots, I_n \} \subset 2^W$ of the minimum set cover problem on $W=[m]$.
We define the following inhomogeneous TLFacP:
\begin{alignat}{3} 
&\text{minimize}& \qquad & x_1 \otimes x_2 \otimes \cdots \otimes x_n, & &	\label{eq:setcover1}\\
&\text{subject to}& \qquad & \bigoplus_{j: I_j \ni i} x_j \geq 1, &\qquad& i \in [m],	\label{eq:setcover2}\\
& & & x_j \geq 0, & &j \in [n].	\label{eq:setcover3}
\end{alignat}
Then, the optimum solution to \eqref{eq:setcover1}--\eqref{eq:setcover3} should be $x \in \{0,1\}^n$.
Any vector $x \in \{0,1\}^n$ satisfying \eqref{eq:setcover2} induces a set cover $\mathscr{I}(x) = \{ I_j : x_j = 1\}$ of $W$.
Conversely, any set cover $\mathscr{I}' \subset \mathscr{I}$ induces a vector $x \in \{0,1\}^n$ satisfying \eqref{eq:setcover2} by setting $x_j = 1$ if $I_j \in \mathscr{I}'$ and $x_j=0$ otherwise.
Moreover, the objective value \eqref{eq:setcover1} is identical to the cardinality of $\mathscr{I}(x)$.
Thus, the TLFacP \eqref{eq:setcover1}--\eqref{eq:setcover3} solves the minimum set cover problem, resulting in the NP-hardness of TLFacP.
\end{proof}

\section{Finding Local Optima}

In the rest of the paper, we consider the TLFacP \eqref{eq:objective}--\eqref{eq:const} defined by $A^\pm = (a_{i,j}^\pm) \in \mathbb{R}_{\max}^{m \times n}$, $C = (c_{k,j}) \in \mathbb{R}_{\max}^{p \times n}$, $\mu^+ \in \mathbb{Z}_{+}^p$, and $\mu^- \in \mathbb{Z}_{+}^n$.
We define $A=(a_{i,j})$, where $a_{i,j}=a^+_{i,j}\oplus a^-_{i,j}$.
Each row of $A$ and $C$ is assumed to contain at least one finite entry; otherwise, such a row can be removed.
The feasible set $S(A^+,A^-)$ is simply denoted as $S$.
To avoid tropical division by $-\infty$, we only consider $x \in S$ such that  $x_j \neq \infty$ for all $j \in [n]$ with $\mu^-_j > 0$.
Taking the observation in Section~3 into account, we assume that $\sum_{k \in [p]} \mu^+_k = \sum_{j \in [n]} \mu^-_j$; otherwise, the problem is not bounded below.

Let $\mathcal{G}$ be the undirected graph with vertex sets $U=\{u_1,u_2,\dots,u_p\}$, $V = [n]$, and $W = \{w_1,w_2,\dots,w_m\}$, and edge set $E = \{ \{u_k,j\} : c_{k,j} \neq -\infty\} \cup \{ \{w_i,j\} : a_{i,j} \neq -\infty \}$.
If $\mathcal{G}$ is not connected, each connected component of $\mathcal{G}$ induces a smaller TLFacP that can be solved independently.
Hence, we assume that $\mathcal{G}$ is connected.

\subsection{Characterization of Local Optima}
First, we focus on the local structure of a feasible solution. 
We say that $x \in S$ is a local optimum if there exists $\delta > 0$ such that all $y \in S$ with $d(x,y) < \delta$ satisfy $f(y) \geq f(x)$.
Since $d(x,y) < \delta$ implies $\supp(x) = \supp(y)$, we can restrict our attention to solutions with fixed support when analyzing local optimality. In particular, by ignoring the columns of $A^+$, $A^-$, and $C$ not contained in the support, we hereafter assume that all entries of $x$ are finite, that is, $x \in \mathbb{R}^n$.

For $x \in S$, we define
\begin{align*}
	\epsilon_A(x) &= \pmin_{i \in [m], j_1,j_2 \in [n]} |(a_{i,j_1} + x_{j_1}) - (a_{i,j_2} + x_{j_2})|,	\\
	\epsilon_C(x) &= \pmin_{k \in [p], j_1,j_2 \in [n]} |(c_{k,j_1} + x_{j_1}) - (c_{k,j_2} + x_{j_2})|,		\\
	\epsilon(x) &= \min(\epsilon_A,\epsilon_C),
\end{align*}
where ``p-min'' stands for the positive-minimum: $\pmin_{i \in [n]} \alpha_i = \min\{ \alpha_i : \alpha_i > 0 \}$.
For any $J \subset [n]$, we observe that $\tilde{f}(\delta) := f(x + \delta\chi_J)$ is a usual affine function of $\delta$ for $0 < \delta < \epsilon_C(x)$.
Moreover, $x + \epsilon_A(x)\chi_J \in S$ implies $x + \delta\chi_J \in S$ for all $0 < \delta < \epsilon_A(x)$.

Theorem~\ref{thm:locmin} below states that the local optimality of $x \in S$ can be checked by vectors $x + \delta \chi_{J}$ for $J \subset [n]$ and small $\delta > 0$.
This fact is well known for the minimization of polyhedral L-convex functions~\cite{Murota2000}.
We expand this result to the minimization with tropical linear inequality constraints.
Indeed, we only use \Cref{prop:lconv} as properties of $f(x)$ in the following proof.

\begin{theorem}	\label{thm:locmin}
If $x \in S$ is not a local optimum, then there exists $J \subset [n]$ such that $x + \delta \chi_{J} \in S$ and $f(x + \delta \chi_{J}) < f(x)$ for all $0 < \delta < \epsilon(x)$.
\end{theorem}

\begin{proof}
If $x \in S$ is not a local optimum, then there exists $\xi \in \mathbb{R}^n$ such that $x + \xi \in S$, $\max_{j \in [n]} |\xi_j| < \epsilon(x)/2$, and $f(x + \xi) < f(x)$.
From \Cref{prop:linsys} and \Cref{prop:lconv}, we assume without loss of generality that $\xi_1 \geq \xi_2 \geq \cdots \geq \xi_n = 0$ by subtracting $\min_{j \in [n]} \xi_j$ from each entry and renumbering the indices.
Then, $\xi_j < \epsilon(x)$ for all $j \in [n]$.

Suppose $0 = \xi_n = \cdots = \xi_{r+1} < \xi_{r}$ and define $y := (\xi_r \otimes x) \oplus' (x + \xi) = x + \xi_r \chi_{[r]}$.
We first prove $y \in S$, that is, 
\begin{align}
	\max_{j \in [n]} (a^+_{i,j} + y_{j}) \geq \max_{j \in [n]} a^-_{i,j} + y_{j}		\label{eq:thmlocmin1}
\end{align}
for any $i \in [m]$.
Since $x \in S$, we have
\begin{align*}
	\max_{j \in [n]} (a^+_{i,j} + x_{j}) \geq a^-_{i,j_1} + x_{j_1}
\end{align*}
for any $j_1 \in [n]$.
If $j_1$ satisfies $\max_{j \in [n]} (a^+_{i,j} + x_j) > a^-_{i,j_1} + x_{j_1}$, then we have 
\begin{align*}
	\max_{j \in [n]} (a^+_{i,j} + x_{j}) \geq a^-_{i,j_1} + x_{j_1} + \epsilon(x)
\end{align*}
according to the definition of $\epsilon(x)$.
Then, considering $y_j \geq x_j$ for all $j \in [n]$ and $y_{j_1} - x_{j_1} \leq \xi_r < \epsilon(x)$, we obtain
\begin{align*}
	\max_{j \in [n]} (a^+_{i,j} + y_{j}) \geq \max_{j \in [n]} (a^+_{i,j} + x_{j}) \geq a^-_{i,j_1} + x_{j_1} + \epsilon(x)
	>  a^-_{i,j_1} + y_{j_1}.
\end{align*}
Next, suppose $j_1$ satisfies $\max_{j \in [n]} (a^+_{i,j} + x_j) = a^-_{i,j_1} + x_{j_1}$.
Recalling $x + \xi \in S$, we obtain 
\begin{align*}
	\max_{j \in [n]} (a^+_{i,j} + x_j + \xi_j) \geq a^-_{i,j_1} + x_{j_1} + \xi_{j_1}.
\end{align*}
Take $j_2 \in [n]$ such that $\max_{j \in [n]} (a^+_{i,j} + x_j + \xi_j) = a^+_{i,j_2} + x_{j_2} + \xi_{j_2}$.
Since $\xi_{j_2} < \epsilon(x)$, we must have $a^+_{i,j_2} + x_{j_2} = \max_{j \in [n]} (a^+_{i,j} + x_j)$, yielding $a^+_{i,j_2} + x_{j_2} = a^-_{i,j_1} + x_{j_1}$.
Combining this equality with 
\begin{align*}
	a^+_{i,j_2} + x_{j_2} + \xi_{j_2} = \max_{j \in [n]} (a^+_{i,j} + x_j + \xi_j) \geq a^-_{i,j_1} + x_{j_1} + \xi_{j_1},
\end{align*}
we obtain $\xi_{j_2} \geq \xi_{j_1}$.
Hence, we have 
\begin{align*}
a^+_{i,j_2} + x_{j_2} + \min(\xi_r,\xi_{j_2}) \geq a^-_{i,j_1} + x_{j_1} + \min(\xi_r,\xi_{j_1}),
\end{align*}
yielding 
\begin{align*}
	a^+_{i,j_2} + y_{j_2} \geq a^-_{i,j_1} + y_{j_1}.
\end{align*}
Thus, we obtain 
\begin{align*}
	\max_{j \in [n]} (a^+_{i,j} + y_{j}) \geq a^-_{i,j_1} + y_{j_1}
\end{align*}
for all $j_1 \in [n]$, leading to \eqref{eq:thmlocmin1}.

By \Cref{prop:linsys}, we observe that $z := (\xi_r \otimes x) \oplus (x + \xi) \in S$, and by \Cref{prop:lconv}, we have 
\begin{align*}
	f(x) + f(x+\xi) = f(\xi_r \otimes x) + f(x+\xi) 	\geq f(y) + f(z).
\end{align*}
Since $f(x+\xi) < f(x)$, we have $f(y) < f(x)$ or $f(z) < f(x)$.
When $f(y) < f(x)$, by recalling $y = x + \xi_r \chi_{[r]}$ and $\xi_r < \epsilon(x)$, we have $f(x + \delta \chi_{[r]}) < f(x)$ for all $\delta$ with $0 < \delta < \epsilon(x)$.
The fact that $x + \delta \chi_{[r]} \in S$ also holds for $0 < \delta < \epsilon(x)$.
Therefore, we can take $J := [r]$.
If $f(z) < f(x)$, we take $\xi' := (- \xi_r) \otimes (z-x)$ instead of $\xi$ and apply the above argument again.
Since the number of positive entries in $\xi'$ is strictly decreased from that of $\xi$, this iteration will eventually stop, thereby determining the desired subset $J$.
\end{proof}

For $x \in S$, a subset $J \subset [n]$ is called a feasible direction at $x$ if $x + \delta \chi_J \in S$ for $0 < \delta < \epsilon(x)$, and $J$ is called a descent direction at $x$ if $f(x + \delta \chi_J) < f(x)$ for $0 < \delta < \epsilon(x)$. The set of all feasible descent directions at $x$ is denoted by $\mathscr{D}(x)$.
The framework of the descent method to find a local optimum is described in \Cref{alg:base}.
In this algorithm, we do not need the whole set $\mathscr{D}(x)$, but we are only required the way to find a subset $J \in \mathscr{D}(x)$ or to detect $\mathscr{D}(x) = \emptyset$.
\begin{algorithm}[H]	\caption{Descent method}	\label{alg:base}
\begin{algorithmic}[1]  
    \STATE Find a feasible solution $x \in S$
    \WHILE{$\mathscr{D}(x) \neq \emptyset$}
    \STATE Take $J \in \mathscr{D}(x)$
    \STATE Determine the step size $\delta > 0$
    \STATE $x \leftarrow x + \delta \chi_J$
    \ENDWHILE
\end{algorithmic}
\end{algorithm}

To determine a feasible descent direction in a combinatorial way, we exploit the tangent digraph $\mathcal{G}(x)$ at $x \in S$. The digraph $\mathcal{G}(x)$ consists of three vertex sets, $U = \{u_1,u_2,\dots,u_p\}$, $V = [n]$, and $W=\{w_1,w_2,\dots,w_m\}$, corresponding to the rows of $C \otimes x$, entries of $x$, and rows of $A^+ \otimes x \geq A^- \otimes x$, respectively. The set of directed edges of $\mathcal{G}(x)$ is $E(x) = E_1(x) \cup E_2(x) \cup E_3(x)$, where 
\begin{align*}
	&E_{1}(x) = \{ (u_k,j) \in U \times V : (C \otimes x)_k = c_{k,j} + x_j \}, \\
	&E_{2}(x) = \{ (w_i,j) \in W \times V : (A \otimes x)_i = a^-_{i,j} + x_j \}, \\
	&E_{3}(x) = \{ (j,w_i) \in V \times W : (A \otimes x)_i = a^+_{i,j} + x_j \}.
\end{align*}
For $J \subset V$, let $\mathcal{N}_U(J,x)$ be the set of vertices in $U$ adjacent to some vertex in $J$, that is, 
\begin{align*}
\mathcal{N}_U(J,x) = \{ u_k \in U : (u_k,j) \in E_1(x) \text{ for some } j \in J\}.
\end{align*}
For $J \subset V$ and $I \subset W$, we similarly define 
\begin{align*}
	\mathcal{N}_W^-(J,x) &= \{ w_i \in W : (w_i,j) \in E_2(x) \text{ for some } j \in J\}, \\
	\mathcal{N}_W^+(J,x) &= \{ w_i \in V : (j,w_i) \in E_3(x) \text{ for some } j \in J\}, \\
	\mathcal{N}_V^-(I,x) &= \{ j \in V : (j,w_i) \in E_3(x) \text{ for some } w_i \in I\}, \\
	\mathcal{N}(J,x) &= \mathcal{N}_U(J,x) \cup \mathcal{N}_W^-(J,x) \cup \mathcal{N}_W^+(J,x).
\end{align*}
We simply denote $\mathcal{N}_U(\{j\},x)$ as $\mathcal{N}_U(j,x)$, and so forth.
Each vertex $u_k \in U$ is adjacent to at least one $j \in V$, and each vertex $w_i \in W$ satisfies $|\mathcal{N}_V^-(w_i,x)| \geq 1$ if $x \in S$.

\begin{figure}
\centering
\includegraphics[width=0.5\textwidth]{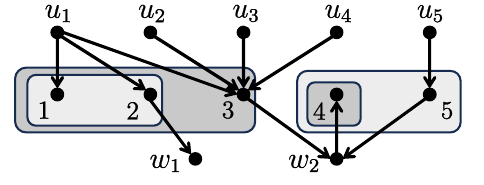}
\caption{Example of tangent digraph. When $\mu^+_k = \mu^-_j = 1$ for all $k$ and $j$, index sets $\{1,2\}$ and $\{4,5\}$ are feasible descent directions. The set $\{1,2,3\}$ is feasible but not descent, whereas $\{4\}$ is descent but not feasible.}	\label{fig:intro2}
\end{figure}

For $x \in S$ and $J \subset V$, we define the rate of increase by
\begin{align*}
	\mu(J,x) := \sum_{u_k \in \mathcal{N}_U(J,x)} \mu^+_k - \sum_{j \in J} \mu^-_j.
\end{align*}
The next theorem characterizes a feasible descent direction in terms of the tangent digraph, see also \Cref{fig:intro2}.
\begin{theorem}		\label{thm:fdd}
For $x \in S$ and $J \subset V$, we have $J \in \mathscr{D}(x)$ if and only if the following properties are satisfied.
\begin{enumerate}
\item For any $j \in J$ and $w_i \in \mathcal{N}^-_W(j,x)$, there exists $j' \in \mathcal{N}^-_V(w_i,x)$ such that $j' \in J$.
\item $\mu(J,x) < 0$.
\end{enumerate}
\end{theorem}

\begin{proof}
First, we focus on the $i$th row of the constraints~\eqref{eq:const}.
If $(A^+ \otimes x)_i > (A^- \otimes x)_i$, then a small change in $x$ does not disrupt the validity of the inequality.
Otherwise, we consider any $j \in [n]$ such that $(A^+ \otimes x)_i = a^-_{i,j} + x_j $.
If a subset $J \subset [n]$ contains $j$, then the vector $x+\delta \chi_J$ for $0 < \delta < \epsilon(x)$ satisfies the $i$th row of~\eqref{eq:const} if and only if there exists $j' \in J$ such that $(A^+ \otimes x)_i = a^+_{i,j'} + x_{j'}$.
Hence, $J$ is a feasible direction if and only if the above argument is valid for all $i \in [m]$.
In terms of the tangent digraph, it is equivalent to the condition that if there exists an edge $(w_i,j) \in E_2(x)$ with $j \in J$, then there also exists an edge $(j',w_i) \in E_3(x)$ for some $j' \in J$.
Thus, Property 1 of the theorem is equivalent to the feasibility of $J$.

If we add $\delta \chi_J$ to $x$, where $0 < \delta < \epsilon(x)$, the value of $(C \otimes x)_k$ is increased by $\delta$ if and only if there exists $j \in J$ such that $(C \otimes x)_k = c_{k,j} + x_j$, that is, $(u_k,j) \in E_1(x)$.
Thus, the change in the objective value is examined as
\begin{align*}
	f(x + \delta \chi_J) - f(x) = \delta\left(\sum_{u_k \in N_U(J,x)} \mu^+_k - \sum_{j \in J} \mu^-_j\right) = \delta\,\mu(J,x).
\end{align*}
This means that $J$ is a descent direction if and only if $\mu(J,x) < 0$.
\end{proof}

For $J \in \mathscr{D}(x)$ and $0 < \delta < \epsilon(x)$, we observe that $J$ is also a feasible descent direction at $x + \delta \chi_J$. We can increase $\delta$ until a new edge appears in the tangent digraph. This observation suggests a reasonable choice of the step size
\begin{align}
\begin{aligned}
	\delta(J,x) := 
	&\min (\pmin_{i \in [m]} (\max_{j \not\in J} (a_{i,j} + x_j) - \max_{j \in J} (a_{i,j} + x_j)),\\
	&\qquad\qquad\qquad \pmin_{k \in [p]} (\max_{j \not\in J} (c_{k,j} + x_j) - \max_{j \in J} (c_{k,j} + x_j))).	\label{eq:stepsize}
\end{aligned}
\end{align}
We formally define $\delta(J,x) = \infty$ if $\max_{j \not\in J} (a_{i,j} + x_j) \leq \max_{j \in J} (a_{i,j} + x_j)$ for all $i \in [m]$ and $\max_{j \not\in J} (c_{k,j} + x_j) \leq \max_{j \in J} (c_{k,j} + x_j)$ for all $k \in [p]$.
If $\delta(J,x) = \infty$ for some $J \in \mathscr{D}(x)$, then TLFacP \eqref{eq:objective}--\eqref{eq:const} is not bounded below because we can make the objective value arbitrarily small by moving along $\chi_J$.

\begin{proposition}		\label{prop:connected}
For any $x \in S$, there exists $x' \in S$ such that $f(x') \leq f(x)$ and $\mathcal{G}(x')$ is connected, or the problem is not bounded below.
\end{proposition}

\begin{proof}
If $\mathcal{G}(x)$ is connected, the assertion is trivial.
Suppose $\mathcal{G}(x)$ is not connected. Let $\mathcal{H}$ be a connected component of $\mathcal{G}(x)$ and $J$ be the set of the vertices of $\mathcal{H}$ belonging to $V$.
Then, both $J$ and $J' := [n] \setminus J$ are feasible directions at $x$ because adding $\delta \chi_J$ or $\delta \chi_{J'}$ for $0 < \delta < \epsilon(x)$ does not change the structure of $\mathcal{G}(x)$. In particular, since $\mu(J,x) = -\mu(J',x)$ according to our assumption $\sum_{k \in [p]} \mu^+_k = \sum_{j \in [n]} \mu^-_j$, either $J$ or $J'$ is a descent direction, or $\mu(J,x) = \mu(J',x) = 0$. Suppose $J$ is a descent direction. If $\delta(J,x) \neq \infty$, then $\mathcal{H}$ will be connected to some other component in $\mathcal{G}(x + \delta(J,x) \chi_J)$. Otherwise, the problem is not bounded below. The case where $J'$ is a descent direction is similar. In the remaining case, where $\mu(J,x) = \mu(J',x) = 0$, at least one of $\delta(J,x)$ and $\delta(J',x)$ is assumed to be finite because $\mathcal{G}$ is connected. Hence, $\mathcal{H}$ is connected to some other component in either $\mathcal{G}(x + \delta(J,x) \chi_J)$ or $\mathcal{G}(x + \delta(J',x) \chi_{J'})$. By applying the above process repeatedly, we obtain $x' \in S$ such that $f(x') \leq f(x)$ and $\mathcal{G}(x')$ is connected, or the problem is not bounded below.
\end{proof}

\begin{remark}
If $\mathcal{G}(x)$ is connected, we can reconstruct the solution $x$ from the tangent digraph $\mathcal{G}(x)$ as follows.
Fix $x_1$ to any finite value. If $(u_k,1) \in E_1(x)$ and $(u_k,j) \in E_1(x)$ for some $k \in [p]$, then we have $c_{k,1}+x_1 = c_{k,j} + x_j$, resulting $x_j = c_{k,1} + x_1 - c_{k,j}$. Other entries of $x$ can be determined recursively in a similar way.
If the TLFacP \eqref{eq:objective}--\eqref{eq:const} is bounded below, then there exists a global optimum $x$ such that $\mathcal{G}(x)$ is connected by \Cref{prop:connected}.
Thus, a global optimum is obtained by enumerating all spanning trees of $\mathcal{G}$ and constructing the corresponding vectors.
Although some vectors may be infeasible, we can find a feasible solution with the minimum objective value among them.
Such a solution is a global optimum, if it exists.
The number of spanning trees of $\mathcal{G}$ rapidly grows as $n$, $m$, and/or $p$ get larger.
If all entries of $A$ and $C$ are finite, then $\mathcal{G}$ is a complete bipartite graph with vertices $V$ and $U \cup W$, which has $n^{m+p-1} (m+p)^{n-1}$ spanning trees, e.g.,~\cite{Jungnickel2013}.
Hence, efficient algorithms to find a local optimum are still meaningful.
\end{remark}

\subsection{Finding Feasible Descent Directions by Minimum $s$-$t$ Cut Problem}

We now describe a method for determining a feasible descent direction. 
We extend the tangent digraph $\mathcal{G}(x)$ to $\overline{\mathcal{G}}(x)$ by appending two vertices $s$ and $t$ and edges $(s,u_k)$ for $k \in [p]$ and $(j,t)$ for $j \in [n]$. 
Subsequently, we assign capacities to the edges of $\overline{\mathcal{G}}(x)$.
The capacities of $(s,u_k)$ and $(j,t)$ are set to $\mu^+_k$ and $\mu^-_j$, respectively, while those of edges in $E(x)$ are $D := \sum_{k\in[p]}\mu^+_k$.

Let $\mathfrak{G}(x)$ be the set of all subgraphs of $\overline{\mathcal{G}}(x)$ obtained by removing edges from $E_3(x)$ so that each $w_i$ has exactly one incoming edge.
An $s$-$t$ cut of $\mathcal{H} \in \mathfrak{G}(x)$ is a pair of vertex sets $S$ and $T$ such that $s \in S$, $t \in T$, $S \cap T = \emptyset$, and $S \cup T = U \cup V \cup W \cup \{s,t\}$.
The capacity $C(S,T)$ of an $s$-$t$ cut $(S,T)$ is the sum of the capacities of all edges $(s_1,t_1)$ of $\mathcal{H}$ with $s_1 \in S$ and $t_1 \in T$.

\begin{theorem}	\label{thm:stcut}
A feasible solution $x \in S$ is a local optimum if and only if the capacity of the minimum $s$-$t$ cut coincides with $D$ for any $\mathcal{H} \in \mathfrak{G}(x)$.
If $x \in S$ is not a local optimum, a feasible descent direction is given by $J = T \cap V$, where $(S, T)$ is the minimum $s$-$t$ cut with $C(S, T) < D$ for some $\mathcal{H} \in \mathfrak{G}(x)$.
\end{theorem}

\begin{proof}
Let $(S,T)$ be the minimum $s$-$t$ cut on $\mathcal{H} \in \mathfrak{G}(x)$ and suppose $C(S,T) < D$.
Since the capacities of edges in $E(x)$ are $D$, these edges from $S$ to $T$ are not included in the cutset.
Thus, we have
\begin{align}
	C(S,T) = \sum_{j \in S \cap V} \mu^-_j + \sum_{u_k \in T \cap U} \mu^+_k. 	\label{eq:mincut1}
\end{align}
If $j \in T \cap V$, then $\mathcal{N}^-_W(j,x) \subset T$. Subsequently, for any $w_i \in \mathcal{N}^-_W(j,x)$, the unique vertex $j' \in \mathcal{N}^-_V(w_i,x)$ is contained in $T$. Thus, $J := T \cap V$ satisfies Property 1 of \Cref{thm:fdd}. 
Additionally, since $\mathcal{N}_U(J,x) \subset T$, we have
\begin{align}
	\sum_{u_k \in \mathcal{N}_U(J,x)} \mu^+_k \leq \sum_{u_k \in T \cap U} \mu^+_k.		\label{eq:mincut2}
\end{align}
By combining \eqref{eq:mincut1} and \eqref{eq:mincut2} and considering $C(S,T) < D = \sum_{j \in [n]} \mu^-_j$, we obtain
\begin{align*}
	\sum_{u_k \in \mathcal{N}_U(J,x)} \mu^+_k 
	< \sum_{j \in [n]} \mu^-_j - \sum_{j \in S \cap V} \mu^-_j
	= \sum_{j \in J} \mu^-_j,
\end{align*}
which implies $\mu(J,x) < 0$. 
Thus, we find $J \in \mathscr{D}(x)$ by \Cref{thm:fdd}, implying that $x$ is not a local optimum.

Conversely, assume that there exists $J \in \mathscr{D}(x)$. 
According to \Cref{thm:fdd}, we choose $\mathcal{H} \in \mathfrak{G}(x)$ such that for any $j \in J$ and $w_i \in \mathcal{N}^-_W(j,x)$ some edge $(j',w_i) \in E_3(x)$ with $j' \in J$ is contained in $\mathcal{H}$. 
In this subgraph, all vertices in $V$ that can reach $J$ are contained in $J$. Hence, $\mathcal{N}_U(J,x)$ is the set of vertices in $U$ that can reach $J$. 
Let $T$ be the union of $\{t\}$ and the set of all vertices that can reach $J$, including $J$ itself and excluding $s$, and let $S$ be the set of the other vertices.
Then, the set of all edges from $S$ to $T$ is $\{ (s,u_k) : u_k \in \mathcal{N}_U(J,x)\} \cup \{ (j,t) : j \not\in J\}$. 
The capacity of the cut is given by
\begin{align*}
	C(S,T) = \sum_{u_k \in \mathcal{N}_U(J,x)} \mu^+_k + \sum_{j \not\in J} \mu^-_j.
\end{align*}
From Theorem~\ref{thm:fdd}, we have $\mu(J,x) = \sum_{u_k \in \mathcal{N}_U(J,x)} \mu^+_k - \sum_{j \in J} \mu^-_j < 0$.
Thus, we obtain
\begin{align*}
	C(S,T) < \sum_{j \in J} \mu^-_j + \sum_{j \not\in J} \mu^-_j = D,
\end{align*}
which implies that the capacity of the minimum $s$-$t$ cut is less than $D$.
The final statement of the theorem is straightforward from the proof.
\end{proof}

By using \Cref{thm:stcut}, we can determine a feasible descent direction in at most $|\mathfrak{G}(x)| \cdot \mathrm{MinCut}(m+n+p+2,|E(x)|+m+n)$ time, where $\mathrm{MinCut}(n',r')$ is the complexity to solve the minimum $s$-$t$ cut problem with $n'$ vertices and $r'$ edges.
The minimum $s$-$t$ cut problem can be efficiently solved; for example, $\mathrm{MinCut}(n',r') = O(n'r')$ by using an algorithm for the maximum flow problem~\cite{Orlin2013}.
On the other hand, the size of $\mathfrak{G}(x)$ may be exponential.

\section{Non-degenerate Problems}

Although it is hard to efficiently find a feasible descent direction in general, we devise a polynomial-time algorithm when the problem is non-degenerate. The degeneracy of the problem is defined in terms of tropical determinants.
The determinant of a tropical matrix $P = (p_{i,j}) \in \mathbb{R}_{\max}^{n \times n}$ is defined by
\begin{align*}
	\det P := \bigoplus_{\sigma \in S_n} \bigotimes_{i=1}^n p_{i,\sigma(i)} = \max_{\sigma \in S_n} \sum_{i=1}^{n} p_{i,\sigma(i)},
\end{align*}
where $S_n$ is the symmetric group of degree $n$.
The tropical determinant is also called the tropical permanent to distinguish it from the determinant in the symmetrized tropical algebra, which admits tropical signs~\cite{Plus1990}.
If at least two permutations attain the maximum on the right-hand side at the same time, then $P$ is called singular.
Let $Q = \begin{psmallmatrix} C \\ A \end{psmallmatrix} \in \mathbb{R}_{\max}^{(p+m) \times n}$. The TLFacP \eqref{eq:objective}--\eqref{eq:const} is called degenerate if some square submatrix $\hat{Q}$ of $Q$ is singular but $\det \hat{Q} \neq -\infty$; otherwise called non-degenerate.

\begin{proposition}		\label{prop:degene}
If the TLFacP \eqref{eq:objective}--\eqref{eq:const} is non-degenerate, then the underlying undirected graph of $\mathcal{G}(x)$ has no cycle for any $x \in S$.
\end{proposition}

\begin{proof}
We identify $\mathcal{G}(x)$ and its underlying undirected graph without any confusion.
For simplicity, we rename vertices $u_1, u_2, \dots, u_p$ and $w_1, w_2, \dots, w_m$ as $\overline{1}, \overline{2}, \dots, \overline{p+m}$.
To show the contraposition, let $(j_1, \overline{\ell_1}, j_2, \overline{\ell_2}, \dots, \overline{\ell_{s}}, j_{s+1})$ be a cycle in $\mathcal{G}(x)$, where $j_{s+1} = j_1$.
Since consecutive vertices in the cycle indicate edges in $\mathcal{G}(x)$, we have
\begin{align*}
	(Q \otimes x)_{\ell_i} = q_{\ell_i,j_i} + x_{j_i} = q_{\ell_i,j_{i+1}} + x_{j_{i+1}}
\end{align*}
for $i \in [s]$. Noting that $x_{j_1} = x_{j_{s+1}}$, we obtain
\begin{align}
	\sum_{i=1}^s q_{\ell_i,j_i} = \sum_{i=1}^s q_{\ell_i,j_{i+1}}.		\label{eq:degene}
\end{align}
Let us consider $s \times s$ submatrix $\hat{Q}$ of $Q$ whose rows and columns are indexed by $L = \{\ell_1, \ell_2, \dots, \ell_s\}$ and $J = \{j_1,j_2, \dots, j_s\}$, respectively.
We define bijections $\sigma_1, \sigma_2 : L \to J$ by $\sigma_1(\ell_i) = j_i$ and $\sigma_2(\ell_i) = j_{i+1}$ for $i \in [s]$.
Then, \eqref{eq:degene} can be rewritten as
\begin{align*}
	\sum_{\ell \in L} q_{\ell,\sigma_1(\ell)} = \sum_{\ell \in L} q_{\ell,\sigma_2(\ell)}.
\end{align*}
For any bijection $\tau: L \to J$, we have
\begin{align*}
	q_{\ell_i, \tau(\ell_i)} + x_{\tau(\ell_i)} \leq (Q \otimes x)_{\ell_i} = q_{\ell_i,j_i} + x_{j_i}
\end{align*}
for $i \in [s]$, which yields
\begin{align*}
	\sum_{\ell \in L} q_{\ell,\tau(\ell)} \leq \sum_{\ell \in L} q_{\ell,\sigma_1(\ell)}.
\end{align*}
Hence, we have
\begin{align*}
	\det \hat{Q} = \sum_{\ell \in L} q_{\ell,\sigma_1(\ell)} = \sum_{\ell \in L} q_{\ell,\sigma_2(\ell)},
\end{align*}
which implies that $\hat{Q}$ is singular. Moreover, we have $\det\hat{Q} \neq -\infty$ because every edge of $\mathcal{G}(x)$ comes from a finite entry of $Q$. Thus, we conclude that the TLFacP \eqref{eq:objective}--\eqref{eq:const} is degenerate.
\end{proof}

\begin{remark}
We observe that non-degenerate TLFacP is still NP-hard, although the homogenization of \eqref{eq:setcover1}--\eqref{eq:setcover3} in Section 3.4 is degenerate.
We modify this problem as follows.
\begin{alignat}{3} 
&\text{minimize}& \qquad & x_1\otimes x_2 \otimes \cdots \otimes x_n, &\quad &	\label{eq:setcover4}\\
&\text{subject to}& \qquad & \bigoplus_{j: I_j \ni i} (-2^{i(n+1)+j}) \otimes x_j \geq -2^{i(n+1)} + N,	& & i \in [m], \label{eq:setcover5}\\
& & & (-2^j) \otimes x_j \geq 0, & &j \in [n],	\label{eq:setcover6}
\end{alignat}
where $N = 2^{(mn+m+n) + \lceil \log_2n \rceil}$.
Note that the total bit length of the coefficients is $O(m^2n^2)$.
The optimal solution $x$ of \eqref{eq:setcover4}--\eqref{eq:setcover6} should satisfy 
\begin{align}
	x_j \in \{ 2^j,\, 2^{(n+1)}(2^j-1) + N,\, 2^{2(n+1)}(2^j-1) + N, \dots, 2^{m(n+1)}(2^j-1) + N\}	\label{eq:setcover7}
\end{align}
for all $j \in [n]$.
Let $x$ be a vector satisfying \eqref{eq:setcover5} and \eqref{eq:setcover7} and define $\mathscr{I}(x) = \{ I_j : x_j \neq 2^j\}$.
Then, $\mathscr{I}(x)$ is a set cover of $[m]$. Conversely, every set cover induces a vector $x$ satisfying \eqref{eq:setcover5} and \eqref{eq:setcover7}.
In addition, by noting that $n2^{m(n+1)}(2^n-1)< n2^{mn+m+n} \leq N$, we have
\begin{align*}
	N \cdot |\mathscr{I}(x)| \leq
	x_1\otimes x_2 \otimes \cdots \otimes x_n <
	N \cdot (|\mathscr{I}(x)| + 1).
\end{align*}
Hence, the optimal solution gives the minimum set cover.
Moreover, the homogenization of \eqref{eq:setcover4}--\eqref{eq:setcover6} is non-degenerate. Indeed, let us consider $Q = (q_{\ell,j}) = \begin{psmallmatrix} C \\ A \end{psmallmatrix}$ defined by $C=(c_{k,j}) \in \mathbb{R}_{\max}^{n \times (n+1)}$ and $A=(a_{i,j}) \in \mathbb{R}_{\max}^{(m+n) \times (n+1)}$ with
\begin{alignat*}{2}
	&c_{j,j} = 0 &\quad& \text{for $j \in [n]$}, \\
	&a_{i,j} = -2^{i(n+1)+j} & &\text{for $i \in I_j$, $i \in [m]$, $j \in [n]$}, \\
	&a_{i,n+1} = -2^{i(n+1)} + N & &\text{for $i \in [m]$}, \\
	&a_{j+m,j} = -2^j & &\text{for $j \in [n]$}, \\
	&a_{j+m,n+1} = 0 & &\text{for $j \in [n]$},
\end{alignat*}
and all other entries $-\infty$. 
It is remarkable that all exponents of $2$ are different except for $N = 2^{(mn+m+n) + \lceil \log_2n \rceil}$.
If the homogenization of \eqref{eq:setcover4}--\eqref{eq:setcover6} is degenerate, then there exist subsets $L \subset [m+2n]$ and $J \subset [n+1]$ with $|L|=|J|$, and bijections $\sigma_1,\sigma_2:L\to J$, such that $\sigma_1(\ell) \neq \sigma_2(\ell)$ for all $\ell \in L$ and 
\begin{align}
	\sum_{\ell \in L} q_{\ell,\sigma_1(\ell)} = \sum_{\ell \in L} q_{\ell,\sigma_2(\ell)}.	\label{eq:nphard1}
\end{align}
Let $2^\nu$ be the smallest power of $2$ appearing in $q_{\ell,\sigma_1(\ell)}$ or $q_{\ell,\sigma_2(\ell)}$ for $\ell \in L$.
Then, $2^\nu$ is included in exactly one side of \eqref{eq:nphard1}.
This implies that one side of \eqref{eq:nphard1} is divisible by $2^{\nu+1}$, but the other is not, which is a contradiction.
Thus, the concerning problem is non-degenerate.
\end{remark}

In the remainder of this paper, we assume that the TLFacP \eqref{eq:objective}--\eqref{eq:const} is non-degenerate. 
Let $x \in S$ and assume by \Cref{prop:connected} that $\mathcal{G}(x)$ is a spanning tree.
If we remove an edge $e$ from $\mathcal{G}(x)$, it is divided into two trees.
Let $\mathcal{T}(e,x)$ be one of such trees that contains the endpoint of $e$ belonging to $U$ or $W$.
For a subgraph $\mathcal{T}$ of $\mathcal{G}(x)$, let $U(\mathcal{T})$, $V(\mathcal{T})$, and $W(\mathcal{T})$ denote the sets of all the vertices of $\mathcal{T}$ belonging to $U$, $V$, and $W$, respectively. 
Additionally, we define $\mu(\mathcal{T})$ as $\mu(\mathcal{T}) := \sum_{u_k \in U(\mathcal{T})} \mu^+_k - \sum_{j \in V(\mathcal{T})} \mu^-_j$.
For $J \subset V$, let $\mathcal{G}(J,x)$ be a subgraph of $\mathcal{G}(x)$ induced by $J \cup \mathcal{N}(J,x)$.

\begin{lemma}		\label{lem:divide}
Let $x \in S$ and $e \in E(x)$.
Then, $J := V(\mathcal{T}(e,x))$ is a feasible direction unless $e = (j,w_i) \in E_3(x)$ and $|\mathcal{N}^-_V(w_i,x)| = 1$.
Furthermore, $\mu(J,x) = \mu(\mathcal{T}(e,x))$.
\end{lemma}

\begin{proof}
Contrarily, suppose $J$ is not a feasible direction.
From Property 1 of \Cref{thm:fdd}, there exist $j \in J$ and $w_i \in \mathcal{N}^-_W(j,x)$ such that $\mathcal{N}^-_V(w_i,x) \cap J = \emptyset$.
This occurs only when a unique edge entering $w_i$ is removed from $\mathcal{G}(x)$.
Moreover, as no edge $(u_k,j)$ with $j \in J$ is removed, we have $\sum_{u_k \in U(\mathcal{T})} \mu^+_k = \sum_{u_k \in \mathcal{N}_U(J)} \mu^+_k$, which leads to $\mu(J,x) = \mu(\mathcal{T}(e,x))$.
\end{proof}

From this observation, we define $E^{\circ}(x)$ as the set of all removal edges in $\mathcal{G}(x)$, that is,
\begin{align*}
	E^{\circ}(x) = E_1(x) \cup E_2(x) \cup \{ (j,w_i) \in E_3(x) : |\mathcal{N}^-_V(w_i,x)|\geq 2\}.
\end{align*}
The next theorem states that we only need to consider the trees obtained by removing one edge from $\mathcal{G}(x)$ when we look for a feasible descent direction.

\begin{theorem}		\label{thm:tree}
If $x \in S$ is not a local optimum, then there exists $e \in E^{\circ}(x)$ such that $V(\mathcal{T}(e,x)) \in \mathscr{D}(x)$.
\end{theorem}

\begin{proof}
Since $x \in S$ is not a local optimum, we have $J \in \mathscr{D}(x)$.
First, we may assume that $\mathcal{G}(J,x)$ is connected.
Indeed, if it is not connected, let $\mathcal{G}_1$ be a connected component and define $J_1 := V(\mathcal{G}_1)$.
As $J$ satisfies Property 1 of \Cref{thm:fdd}, for any $j \in J$ and $w_i \in \mathcal{N}^-_W(j,x)$ there exists $j' \in \mathcal{N}^-_W(w_i,x) \cap J$. 
From the connectivity of $\mathcal{G}_1$, if $j \in J_1$, then $j' \in J_1$. Hence, $J_1$ satisfies Property 1 of \Cref{thm:fdd}, which means that $J_1$ is a feasible direction at $x$.
Similarly, $J_2 := [n] \setminus J_1$ is also feasible.
Noting that $\mathcal{N}_U(J_1,x) \cap \mathcal{N}_U(J_2,x) = \emptyset$, we have $\mu(J,x) = \mu(J_1,x) + \mu(J_2,x)$. 
Thus, recalling that $\mu(J) < 0$, we have $\mu(J_1,x) < 0$ or $\mu(J_2,x) < 0$, implying $J_1 \in \mathscr{D}(x)$ or $J_2 \in \mathscr{D}(x)$. Repeating this process, we obtain a feasible descent direction that induces a connected subgraph.

Next, let $\mathcal{H}$ be the subgraph of $\mathcal{G}(x)$ obtained by removing all the vertices in $\mathcal{G}(J,x)$ together with all the edges incident to them. Let $\mathcal{H}_1, \mathcal{H}_2, \dots, \mathcal{H}_s$ be the connected components of $\mathcal{H}$.
For each $\ell \in [s]$, there exists a unique edge $e_\ell$ that connects $\mathcal{H}_\ell$ to $\mathcal{G}(J,x)$.
Since $\mathcal{N}(J,x)$ is contained in $\mathcal{G}(J,x)$ and $J$ is a feasible direction, each edge $e_\ell$ takes one of the following forms:
\begin{enumerate}
\item $e_\ell = (u_k,j) \in E_1(x)$ with $u_k \in U(\mathcal{G}(J,x))$ and $j \in V(\mathcal{H}_\ell)$;
\item $e_\ell = (w_i,j) \in E_2(x)$ with $w_i \in W(\mathcal{G}(J,x))$ and $j \in V(\mathcal{H}_\ell)$;
\item $e_\ell = (j,w_i) \in E_3(x)$ with $w_i \in W(\mathcal{G}(J,x))$, $j \in V(\mathcal{H}_\ell)$, and $\mathcal{N}^-_V(w_i,x) \cap J \neq \emptyset$.
\end{enumerate}
In the last case, we have
\begin{align*}
	|\mathcal{N}^-_V(w_i,x)| \geq |\mathcal{N}^-_V(w_i,x) \cap J| + |\{j\}| \geq 2.
\end{align*}
Thus, $e_{\ell} \in E^{\circ}(x)$ for all $\ell \in [s]$.
We can take a component $\mathcal{H}_\ell$ satisfying $\mu(\mathcal{H}_\ell) > 0$.
Indeed, if $\mu(\mathcal{H}_\ell) \leq 0$ for all $\ell \in [s]$, then, by noting that $\mu(\mathcal{G}(J,x)) = \mu(J,x) < 0$, we have
\begin{align*}
	\mu(\mathcal{G}(J,x)) + \sum_{\ell \in [s]} \mu(\mathcal{H}_\ell) < 0.
\end{align*}
However, the left-hand side is identical to $\sum_{u_k \in U} \mu^+_k - \sum_{j \in V} \mu^-_j = 0$, which is a contradiction.
When we remove $e_\ell$ from $\mathcal{G}(x)$, we obtain two connected components: $\mathcal{H}_\ell$ and $\mathcal{T}(e_\ell,x)$. Then, $J^* := V(\mathcal{T}(e_\ell,x))$ is a feasible direction according to \Cref{lem:divide}. In addition, $\mu(J^*,x) = \mu(\mathcal{T}(e_\ell,x)) = -\mu(\mathcal{H}_\ell) < 0$.
Thus, we have found $J^* \in \mathscr{D}(x)$ as desired.
\end{proof}

\begin{example}
Let us consider a TLFacP \eqref{eq:objective}--\eqref{eq:const} with
\begin{gather*}
C = \begin{pmatrix} 2 & 0 & 0 \\ 1 & 2 & 0 \\ 0 & 0 & 7 \end{pmatrix},\
A^+ = \begin{pmatrix} -5 & -\infty & 0 \\ -1 & -\infty & -\infty \\ -\infty &-\infty & 0 \\ -\infty & -1 & -\infty \end{pmatrix},\
A^- = \begin{pmatrix} -\infty & -3 & -\infty \\ -\infty & -2 & 0 \\ -8 & -5 & -\infty \\ -6 & -\infty  & 0 \end{pmatrix}, \\
\mu^+ = \begin{pmatrix} 1 \\ 1 \\ 2 \end{pmatrix},\
\mu^- = \begin{pmatrix} 1 \\ 2 \\ 1 \end{pmatrix},
\end{gather*}
which is illustrated in Figure~\ref{fig:intro1}.
Take an initial vector $x^{(0)} = (1,2,0)^\top$. Then, the tangent digraph $\mathcal{G}(x^{(0)})$ is defined by the edge sets 
\begin{align*}
	E_1(x^{(0)}) &= \{ (u_1,1), (u_2,2), (u_3,3) \}, \\ E_2(x^{(0)}) &= \{(w_2,2),(w_2,3)\}, \\ E_3(x^{(0)}) &= \{(3,w_1),(1,w_2),(3,w_3),(2,w_4)\},
\end{align*}
see Figure~\ref{fig:tangent2} (left). 
Removing $e = (w_2,3) \in E_2(x^{(0)})$ from $\mathcal{G}(x^{(0)})$, we obtain a tree $\mathcal{T}(e,x^{(0)})$ with the edge set
\begin{align*}
	\{(u_1,1),(u_2,2),(w_2,2),(1,w_2),(2,w_4)\}.
\end{align*}
By setting $J := V(\mathcal{T}(e,x^{(0)})) = \{1,2\}$, we can verify that
\begin{align*}
	\mu(J,x^{(0)}) = (1+1) - (1+2) = -1 < 0,
\end{align*}
which means $J$ is a descent direction. We also see that $\delta(J,x^{(0)}) = 1$. Hence, we update the vector as
\begin{align*}
	x^{(1)} := x^{(0)} + \chi_J = (2,3,0)^\top.
\end{align*}
The tangent digraph $\mathcal{G}(x^{(1)})$ is defined by the edge sets
\begin{align*}
	E_1(x^{(1)}) &= \{ (u_1,1), (u_2,2), (u_3,3) \}, \\ E_2(x^{(1)}) &= \{(w_1,2),(w_2,2)\}, \\ E_3(x^{(1)}) &= \{(3,w_1),(1,w_2),(3,w_3),(2,w_4)\},
\end{align*}
see Figure~\ref{fig:tangent2} (right). We cannot obtain a feasible descent direction by removing any edges of $\mathcal{G}(x^{(1)})$.
Thus, $x^{(1)}$ is a local optimum of this TLFacP.
\end{example}

\begin{figure}
\centering
\includegraphics[width=0.85\textwidth]{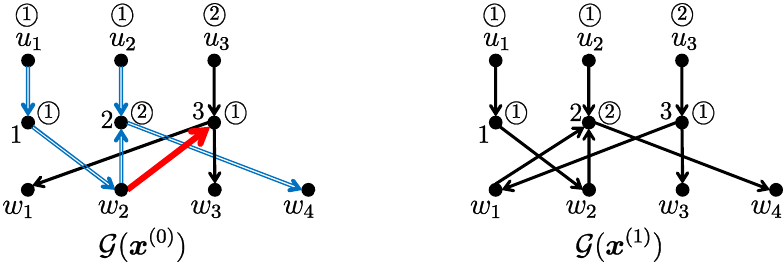}
\caption{Tangent digraphs $\mathcal{G}(x^{(0)})$ and $\mathcal{G}(x^{(1)})$. Numbers in circles indicate the values $\mu^+_k,\, k=1,2,3$ and $\mu^-_j,\,j=1,2,3$. The red bold arrow represents the edge removed from $\mathcal{G}(x^{(0)})$, and the blue double-lined arrows represent the edges of $\mathcal{T}(e,x^{(0)})$.}		\label{fig:tangent2}
\end{figure}

We analyze the computational complexity of \Cref{alg:base} with step size~\eqref{eq:stepsize} when the TLFacP \eqref{eq:objective}--\eqref{eq:const} is non-degenerate and \Cref{thm:tree} is adopted to find $J \in \mathscr{D}(x)$.
We observe that $\mathcal{G}(x)$ has $(m+n+p-1)$ edges when it is a spanning tree.
Starting from the leaves of the tree, we compute $\mu(\mathcal{T}(e,x))$ for all edges $e$ in $O(m+n+p)$ time.
Once we find $J \in \mathscr{D}(x)$, the step size $\delta(J,x)$ is computed in $O(r_A+r_C)$ time by evaluating the finite terms in~\eqref{eq:stepsize}, where $r_A$ and $r_C$ are the numbers of finite entries in $A$ and $C$, respectively.
Thus, the computational complexity of each iteration of lines 3--5 is at most $O(r_A+r_C)$.

\begin{remark}
The descent method proposed in this study can also be applied to a TLP because it is a special case of a TLFacP. The tropical simplex method in~\cite{Allamigeon2015} deals with non-degenerate TLP and has complexity $O(n(m+n))$ for each pivot step.
This complexity is not directly compared to our descent method applied to a TLP because a pivot step in the tropical simplex method may contain several descent steps.
\end{remark}

We further assume that all the finite entries of $A^+$, $A^-$, and $C$ are integers with maximum absolute value $M$. 
An initial solution $x \in S$ such that $\{j \in [n] : \mu^-_j > 0\} \subset \supp(x)$ can be found in at most $O((m+n)r_AM)$ time by solving the tropical linear system as a mean-payoff game~\cite{Brim2011}. After deleting the columns not contained in $\supp(x)$, this solution can be translated to a ``spanning tree'' solution as in \Cref{prop:connected}, which requires at most $(n-1)$ times computation of $\delta(J,x)$, resulting in at most $O(n(r_A+r_C))$ operations in total.
If $\mathcal{G}(x)$ is a spanning tree, then the range of the solution $x$, that is, $\max_{j \in [n]} x_j - \min_{j \in [n]} x_j$, is at most $2(n-1)M$ because any two vertices in $V$ are connected by at most $2(n-1)$ edges. Hence, the objective value $f(x)$ is estimated as
\begin{align*}
	f(x) \leq \sum_{k \in [p]} \mu^+_k (\max_{j \in [n]} c_{k,j} + \max_{j \in [n]} x_j) - \sum_{j \in [n]} \mu^-_j \min_{j' \in [n]} x_{j'}
	\leq (2n-1)MD.
\end{align*}
On the other hand, if a local optimum $x^*$ exists, then $f(x^*)$ is bounded as
\begin{align*}
	f(x^*) \geq \sum_{p \in [p]} \mu^+_k (\min_{j \in [n]} c_{k,j} + \min_{j \in [n]} x^*_j) - \sum_{j \in [n]} \mu^-_j \max_{j' \in [n]} x^*_{j'}
	\geq -(2n-1)MD
\end{align*}
because $\mathcal{G}(x^*)$ is also assumed to be a spanning tree.
Each iteration of lines 3--5 of \Cref{alg:base} decreases the value $f(x)$ by at least one if all the finite entries of $A^+$, $A^-$, and $C$ are integers. 
Hence, the algorithm finds a local optimum or detects that the problem is unbounded within $(4n-2)MD$ iterations.
Since each iteration can be done in $O(r_A+r_C)$ time, the total complexity is $O((m+n)(r_A+r_C)MD)$.

\section{Conclusions}

In this study, we consider a tropical linearly factorized programming (TLFacP).
We propose a combinatorial descent method to find a local optimum.
As any local optimum is represented by a connected tangent digraph, we can determine the global minimum by considering all such digraphs. Although the number of possible connected digraphs is finite, it rapidly grows as the size of the problem gets larger.
Hence, the descent method based on a local search is still useful.
The combinatorial descent algorithm proposed in this study is pseudo-polynomial time if the problem is non-degenerate and the finite entries of the input matrices are integers.
The computational complexity of this polynomial is not so large compared to that of the feasibility problem.
Future studies should focus on developing an effective algorithm for a degenerate TLFacP.

\section*{Acknowledgments}
This work is supported by JSPS KAKENHI Grant No.~22K13964.

\bibliographystyle{abbrv}
\bibliography{TropLFPref_arxiv}
\end{document}